\documentclass[journal,twoside,web]{ieeecolor}

\usepackage{cite}
\usepackage{amsmath,amssymb,amsfonts}
\usepackage{hyperref}
\usepackage{epstopdf}
\usepackage{textcomp}
\usepackage{graphicx,color}
\usepackage{mathrsfs}
\usepackage[vlined,ruled]{algorithm2e}
\usepackage{url}
\usepackage{color}
\usepackage{dsfont}
\usepackage{bbm}
\usepackage{booktabs}
\usepackage{array}
\usepackage[table]{xcolor}
\usepackage{yfonts}
\usepackage{cleveref} 
\usepackage{tikz}
\usepackage{pgfplots}
\usepackage{multirow}
\usepackage{textcomp}
\usepackage{tabularx}
\usepackage{placeins}
\usepackage{float}
\usepackage{nccmath}
\usepackage{accents}
\usepackage{multicol}
\usepackage{etoolbox, refcount}
\usepackage{chngcntr}
\usepackage{apptools}
\usepackage{relsize}
\usepackage[font=footnotesize]{caption}
\usepackage[font=footnotesize]{subcaption}
\usepackage{textcomp}

\newtheorem{theorem}{Theorem}[section]

\newtheorem{proposition}[theorem]{Proposition}

\newtheorem{remark}[theorem]{Remark}

\definecolor{subsectioncolor}{RGB}{1,1,0}

\newcommand{\setdef}[2]{\{#1 \; : \; #2\}}

\newcommand{\until}[1]{\{1,\dots,#1\}}

\newcommand{\Uc}{\mathcal{U}}

\newcommand{\real}{\mathbb{R}}

\newcommand{\integerspos}{\mathbb{Z}_{> 0}}

\newcommand{\ubar}[1]{\underaccent{\bar}{#1}}

\newcommand{\Sc}{\mathcal{S}}
\newcommand{\Kc}{\mathcal{K}}
\newcommand{\Zc}{\mathcal{Z}}
\newcommand{\Pc}{\mathcal{P}}
\newcommand{\Mc}{\mathcal{M}}
\newcommand{\Ic}{\mathcal{I}}
\newcommand{\Nc}{\mathcal{N}}

\newcommand{\argmax}[2] {\mathrm{arg}\max_{#1}#2}

\DeclareSymbolFont{bbold}{U}{bbold}{m}{n}
\DeclareSymbolFontAlphabet{\mathbbold}{bbold}

\newcommand{\norm}[1]{\left\lVert#1\right\rVert}

\newcommand{\CLC}{\operatorname{CLC}}
\newcommand{\CBC}{\operatorname{CBC}}
\newcommand{\CVaR}{\operatorname{CVaR}}

\newcommand{\oprocendsymbol}{\hbox{$\bullet$}}
\newcommand{\oprocend}{\relax\ifmmode\else\unskip\hfill\fi\oprocendsymbol}

\renewcommand{\theenumi}{(\roman{enumi})}

\newcommand*{\QEDA}{\hfill\ensuremath{\blacksquare}}%

\newcommand{\xqed}[1]{%
  \leavevmode\unskip\penalty9999 \hbox{}\nobreak\hfill
  \quad\hbox{#1}}
\newcommand{\demo}{\xqed{$\bullet$}}

\newcounter{countitems}
\newcounter{nextitemizecount}
\newcommand{\setupcountitems}{%
  \stepcounter{nextitemizecount}%
  \setcounter{countitems}{0}%
  \preto\item{\stepcounter{countitems}}%
}
\makeatletter
\newcommand{\computecountitems}{%
  \edef\@currentlabel{\number\c@countitems}%
  \label{countitems@\number\numexpr\value{nextitemizecount}-1\relax}%
}
\newcommand{\nextitemizecount}{%
  \getrefnumber{countitems@\number\c@nextitemizecount}%
}
\newcommand{\previtemizecount}{%
  \getrefnumber{countitems@\number\numexpr\value{nextitemizecount}-1\relax}%
}
\makeatother    
{\end{itemize}%
\unskip\computecountitems\ifnumcomp{\previtemizecount}{>}{4}{\end{multicols}}{}}

\newcommand{\longthmtitle}[1]{\mbox{}\emph{(#1):}}

\renewcommand{\baselinestretch}{.985}
\newcommand{\comment}[1]{} 
\newcolumntype{P}[1]{>{\centering\arraybackslash}p{#1}}
\allowdisplaybreaks

\title{ \huge \bf Feasibility Analysis and Regularity Characterization \\ of Distributionally Robust Safe Stabilizing Controllers}

\author{ \parbox{\linewidth}{ \centering Pol Mestres \qquad Kehan Long \qquad Nikolay Atanasov \qquad Jorge Cort{\'e}s }%
\thanks{The authors are with the Contextual Robotics Institute, UC San Diego (e-mails: \{pomestre,k3long,natanasov,cortes\}@ucsd.edu).}%
\thanks{The authors gratefully acknowledge support from NSF under grants IIS-2007141 and CCF-2112665.}%
}

\begin{document}

\maketitle
\thispagestyle{empty}
\pagestyle{empty}

\begin{abstract}
This paper studies the well-posedness and regularity of safe stabilizing optimization-based controllers for control-affine systems in the presence of model uncertainty.
When the system dynamics contain unknown parameters, a finite set of samples can be used to formulate distributionally robust versions of  control barrier function and control Lyapunov function constraints. Control synthesis with such distributionally robust constraints can be achieved by solving a (convex) second-order cone program (SOCP). We provide one necessary and two sufficient conditions to check the feasibility of such optimization problems, characterize their computational complexity and numerically show that they are significantly faster to check than direct use of SOCP solvers.
Finally, we also analyze the regularity of the resulting control laws.
\end{abstract}


\section{Introduction}

\IEEEPARstart{R}{ecent} years have seen increasing deployment of control systems and robots to aid transportation, warehouse management, and home automation. In these applications, it is crucial to implement controllers with provable safety and stability guarantees despite uncertainty in the system models and operational conditions. Recent work~\cite{KL-VD-ML-JC-NA:22-ral,FC-JJC-BZ-CJT-KS:21,AJT-VDD-SD-BR-YY-ADA:21,PM-JC:23-auto,AL-AB-ND-SH:23,PS-MJ-EH:22}
%
%
tackles this when some prior information about the uncertainty is known. Instead, here we rely on a line of work initiated in~\cite{KL-YY-JC-NA:23-acc} 
that circumvents the need for knowledge about the uncertainty distribution and uses only uncertainty samples to formulate distributionally robust constraints for control synthesis. This approach is robust to distributional shift at deployment time and enjoys provable out-of-sample performance. However, it also introduces several challenges, which we focus on here: the characterization of the quality and number of uncertainty samples needed to guarantee the feasibility of the safety and stability constraints, and the study of the regularity properties of the resulting controllers.

\subsubsection*{Literature Review}
Control Lyapunov functions (CLFs) \cite{EDS:98} are a well-established tool to design stabilizing controllers for nonlinear systems. More recently, control barrier functions (CBFs) \cite{ADA-SC-ME-GN-KS-PT:19} have gained popularity as a tool to render a desired subset of the system state space safe. If the system is control affine,  CLF and CBF constraints are linear in the control input and can be incorporated in a quadratic program (QP) \cite{ADA-XX-JWG-PT:17} that, if feasible, can be solved efficiently to obtain control inputs guaranteeing safety and stability. Recent work has explored alternative optimization formulations when the system model is uncertain. Under the assumption that the uncertainty follows a Gaussian Process (GP) or satisfies worst-case bounds,\cite{FC-JJC-BZ-CJT-KS:21,FC-JJC-BZ-CJT-KS:21-acc,KL-CQ-JC-NA:21-ral,KL-VD-ML-JC-NA:22-ral,AJT-VDD-SD-BR-YY-ADA:21,AL-AB-ND-SH:23}
%
%
formulate second-order cone constraints
that can be used to design controllers achieving safe stabilization of the true system. The paper \cite{PM-JC:23-auto} gives sufficient conditions for the feasibility of such second-order cone constraints. Our work here is closely related to~\cite{KL-YY-JC-NA:23-acc}, which leverages ideas from distributionally robust optimization (DRO) \cite{AB-LEG-AN:09,PME-DK:18} to model the uncertainty. The DRO framework constructs an \textit{ambiguity set} of probability distributions that contains the true (unknown) one with high confidence. Such ambiguity sets are constructed with only finitely many samples and are used to formulate distributionally robust versions of the control design problem.

\subsubsection*{Statement of Contributions}

We study the problem of safe stabilization of control-affine systems under uncertainty. We assume that the distribution of the uncertainty is unknown and formulate a second-order cone program (SOCP) using distributionally robust versions of the CLF and CBF constraints constructed on the basis of uncertainty samples.
Our first contribution is the derivation of a necessary condition and two sufficient conditions for the feasibility of the optimization problem. 
We characterize the computational complexity of these conditions and show that, for a large number of samples, it is significantly smaller than solving the SOCP directly, which makes them useful to efficiently check whether the problem is feasible without having to solve it.
Our first sufficient condition is dependent on the quality of the uncertainty samples but is limited to a single control objective. Our second sufficient condition is only dependent on the number of samples but can be used for any number of constraints.
Our final contribution shows that the solution of this distributionally robust optimization problem is point-Lipschitz, and hence continuous, which means that solutions of the closed-loop system are guaranteed to exist and the controller obtained from it can be implemented
without inducing chattering.

\section{Preliminaries}
We review distrib. robust chance-constrained programs and control Lyapunov and barrier functions under~uncertainty.

\subsection{Distributionally Robust Chance Constrained Programs}\label{sec:drccp}


Given a random vector $\boldsymbol{\xi}$ following distribution $\mathbb{P}^*$ supported on set $\Xi\subseteq\real^k$ and a closed convex set $\Zc\subset\real^n$, let $G :\Zc\times\Xi\to\real$ define a probabilistic constraint $G(z,\boldsymbol{\xi})\leq0$. We are interested in satisfying this constraint with a prescribed confidence $1-\epsilon$, with $\epsilon\in(0,1)$, while minimizing a convex objective function $c:\Zc\to\real$. To achieve this\footnote{We denote by $\mathbb{Z}_{>0}$, $\real$ and $\real_{\geq0}$ the set of positive integers, real, and nonnegative real numbers, resp. We denote by $\textbf{0}_{n}$ the $n$-dimensional zero vector. We write $\partial \Sc$ for the boundary of the set $\Sc$. Given $N\in\mathbb{Z}_{>0}$, we denote $[N]=\{1,\hdots,N\}$. Given $x\in\real^n$, $\norm{x}$ denotes the Euclidean norm of $x$.
For $x \in \real$, we define $(x)_{+}=\max(x,0)$.
A function $\beta:\real_{\geq0}\to\real$ is of class $\Kc_{\infty}$ if $\beta(0)=0$, $\beta$ is strictly increasing and $\lim\limits_{t\to\infty}\beta(t)=\infty$. A function $V:\real^n\to\real$ is positive definite if $V(0)=0$ and $V(x) > 0$ for all $x\neq0$, and proper in a set $\Gamma$ if $\setdef{x\in\Gamma}{V(x)\leq c}$ is compact for any $c\geq0$. 
Given an $m\times n$ matrix $A$ and two integers $i, j$ such that $1\leq i < j\leq m$, $A_{i:j}$ denotes the $(j-i+1)\times n$ matrix obtained by selecting the rows from $i$ to $j$ of $A$.
A function $f:\real^n\to\real^q$ is point-Lipschitz at a point $x_0\in\real^n$ if there exists a neighborhood $U$ of $x_0$ and a constant $L_{x_0}>0$ such that $\norm{f(x)-f(x_0)}\leq L_{x_0}\norm{x-x_0}$ for all $x\in U$.}, define the chance-constrained program:
\begin{align}\label{eq:original-chance-constraint}
&\min\limits_{z\in\Zc} c(z) \\
\notag
\text{s.t.} & \ \mathbb{P}^*(G(z,\boldsymbol{\xi})\leq0)\geq 1-\epsilon.
\end{align}
The feasible set of \eqref{eq:original-chance-constraint} is not convex in general. Nemirosvski and Shapiro \cite[Section 2]{AN-AS:06} propose a convex approximation of the feasible set of \eqref{eq:original-chance-constraint} by replacing the chance constraint with a conditional value-at-risk ($\CVaR$) constraint.
CVaR of $G(z,\boldsymbol{\xi})$ can be formulated as the following convex program:
%
\begin{align}\label{eq:cvar-convex}
\CVaR_{1-\epsilon}^{\mathbb{P}^*}(G(z,\boldsymbol{\xi})) :=\inf\limits_{t\in\real}[\epsilon^{-1} \mathbb{E}_{\mathbb{P}^*}[(G(z,\boldsymbol{\xi})+t)_{+}] \!-\! t].
\end{align}
The resulting problem
\begin{align}\label{eq:cvar-pstar}
&\min\limits_{z\in\Zc} c(z) \\
\notag
\text{s.t.} & \ \CVaR_{1-\epsilon}^{\mathbb{P}^*}(G(z,\boldsymbol{\xi}))\leq0,
\end{align}
is convex and its feasible set is contained in that of \eqref{eq:original-chance-constraint}.

Both \eqref{eq:original-chance-constraint} and~\eqref{eq:cvar-pstar} assume that $\mathbb{P}^*$ is known. Instead, suppose that it is unknown and we only have access to samples $\{ \boldsymbol{\xi}_i\}_{i\in[N]} $ from $\mathbb{P}^*$. We describe a way of constructing a set of distributions that could have generated the samples.
Let $\Pc_p(\Xi)$ be the set of probability measures with finite $p$-th moment supported on $\Xi$. 
Let $\hat{\mathbb{P}}_N:=\frac{1}{N}\sum_{i=1}^N \delta_{\boldsymbol{\xi}_i}$ be the empirical distribution constructed from the samples $\{\boldsymbol{\xi}_i\}_{i=1}^N$. Let $W_p$ be the $p$-Wasserstein distance~\cite[Definition 3.1]{PME-DK:18}
between two probability measures in $\Pc_p(\Xi)$ and let $\Mc_N^r := \setdef{\mu\in\Pc_p(\Xi)}{W_p(\mu,\hat{\mathbb{P}}_N)\leq r}$ be the ball of radius $r$ centered at $\hat{\mathbb{P}}_N$. We define a distributionally robust chance-constrained program:
\begin{align}\label{eq:drccp}
&\min\limits_{z\in\Zc} c(z) \\
\notag
\text{s.t.} &\inf_{\mathbb{P}\in\Mc_N^r} \mathbb{P}(G(z,\boldsymbol{\xi})\leq0)\geq 1-\epsilon.
\end{align}
We can use CVaR to obtain a convex conservative approximation of~\eqref{eq:drccp}:
\begin{align}\label{eq:cvar-drccp}
&\min_{z\in\Zc} c(z) \\
\notag
\text{s.t.} &\sup_{\mathbb{P}\in\Mc_N^r} \CVaR_{1-\epsilon}^{\mathbb{P}}(G(z,\boldsymbol{\xi}))\leq 0.
\end{align}
If~\eqref{eq:cvar-drccp} is feasible, then~\eqref{eq:drccp} is also feasible~\cite[Section 2]{AN-AS:06}.

We say that a distribution $\mathbb{P}$ is light-tailed if there exists $a>0$ such that $A:=\mathbb{E}_{\mathbb{P}}[\exp{\norm{\boldsymbol{\xi}}^a}]=\int_{\Xi}\exp{\norm{\boldsymbol{\xi}}^a}\mathbb{P}(d\boldsymbol{\xi})<\infty$. If $\mathbb{P}^*$ is light-tailed, the following observation specifies how the radius of $\Mc_N^r$ should be selected so that the true distribution lies in the ball with high confidence.

\begin{remark}\longthmtitle{Choice of Wasserstein ball radius}\label{rem:choice-r}
    If the true distribution $\mathbb{P}^*$ is light-tailed, 
    the choice of $r=r_N(\bar{\epsilon}) $ given in~\cite[Theorem 3.5]{PME-DK:18},
    \begin{align}
    \label{eq:wasserstein_r_guarantee}
        r_N(\bar{\epsilon}) = \begin{cases}
            (\frac{\log(c_1\bar{\epsilon}^{-1})}{c_2 N})^{\frac{1}{\max\{k,2\}}} \quad &\text{if} \ N\geq \frac{\log(c_1\bar{\epsilon}^{-1})}{c_2}, \\
            (\frac{\log(c_1\bar{\epsilon}^{-1})}{c_2 N})^{\frac{1}{a}} \quad &\text{else},
        \end{cases}
    \end{align}
    where $c_1, c_2$ and $a$ are positive constants that only depend on 
    $a$, $A$ and $k$ (cf.~\cite[Theorem 3.4]{PME-DK:18}), ensures that the ball $\Mc_N^{r_{N}(\bar{\epsilon})}$ contains $\mathbb{P}^*$ with probability at least $1-\bar{\epsilon}$. Then, a solution $z^*$ of~ \eqref{eq:cvar-drccp} satisfies the constraint $\CVaR_{1-\epsilon}^{\mathbb{P}^*}(G(z^*,\boldsymbol{\xi}))\leq0$
    with probability at least $1-\bar{\epsilon}$. 
    Note that $c_1$, $c_2$ and $a$ can be computed by knowing the class of distributions to which $\mathbb{P}^*$ belongs to, without having actual knowledge of $\mathbb{P}^*$. 
    If exact values are not known, but upper and lower bounds are, these can be used instead to compute an upper bound of $r_N(\bar{\epsilon})$.
    \demo
\end{remark}

\begin{remark}\longthmtitle{Choice of $\epsilon$}\label{rem:epsilon-assumption}
   The parameter $\epsilon$ determines the confidence level $1-\epsilon$ for constraint satisfaction. Throughout the paper, we assume $\epsilon\leq\frac{1}{N}$, albeit results are valid generally, with explicit expressions becoming more involved.
    \demo
\end{remark}

\subsection{Distributionally Robust Safety and Stability}\label{subsec:dr-safety-stability}
%
%
The notions of CLF~\cite{EDS:98} and CBF~\cite{ADA-SC-ME-GN-KS-PT:19} can be used to design controllers in uncertainty-free systems that enforce stability and safety, resp. Here we extend these notions for systems with uncertainty in the dynamics. Consider a nominal model $F$
and a linear combination of $k$ perturbations,
\begin{align}\label{eq:uncertain-model}
\dot{x} &= (F(x) + \sum_{j=1}^k W_j(x)\xi_j)\ubar{u},
\end{align}
where for $1\leq j \leq k$, $W_j(x)\in\real^{n\times(m+1)}$ denotes known model perturbations, and $\xi_j\in\real$ denotes the corresponding unknown weight, and $\ubar{u}=[1; u] \in \ubar{\Uc}:=\{1\}\times\real^m$.
We let
$\boldsymbol{\xi}=[\xi_{1},\xi_{2},\hdots,\xi_{k}]^T  \in \Xi\subseteq\real^{k}$. 
We assume that $\boldsymbol{\xi}$ follows an unknown distribution $\mathbb{P}^*$
but a set of samples $\{\boldsymbol{\xi}_i\}_{i=1}^N$ is available. We are interested in extending the notions of CLF \cite{EDS:98} and CBF \cite{ADA-SC-ME-GN-KS-PT:19} for systems of the form~\eqref{eq:uncertain-model}.
To do so, note that as shown in \cite[Section IV]{KL-YY-JC-NA:23-acc}, the CBF condition for a system of the form \eqref{eq:uncertain-model} and a function $h:\real^n\to\real$ reads as $\CBC(x,\ubar{u},\boldsymbol{\xi}):=\ubar{u}^T q_h(x)+\ubar{u}^T R_h(x) \boldsymbol{\xi} \geq 0$, where the exact forms of $q_h$ and $R_h$ are given in \cite[Section IV]{KL-YY-JC-NA:23-acc} and depend on $h$ and its gradient. Now, since $\boldsymbol{\xi}$ follows a distribution $\mathbb{P}^*$, we extend the definition of CBF by requiring that for all $x$ in the safe set, there exist $\ubar{u}\in\ubar{\Uc}$ such that
\begin{align}\label{eq:p-clc}
    \mathbb{P}^*(\CBC(x,\ubar{u},\boldsymbol{\xi})\geq 0) \geq 1-\epsilon.
\end{align}
The CLF condition for~\eqref{eq:uncertain-model} takes a similar form and is written as $\CLC(x,\ubar{u},\boldsymbol{\xi})\leq0$ (cf.~\cite[Section IV]{KL-YY-JC-NA:23-acc}). As shown in Section~\ref{sec:drccp}, $\CVaR$ can be used as a convex approximation of~\eqref{eq:p-clc} and its analogue with $\CLC$. We use
\begin{subequations}\label{eq:cvar-cbc-clc}
\begin{align}
    \CVaR_{1-\epsilon}^{\mathbb{P}^*}(\CBC(x,\ubar{u},\boldsymbol{\xi}))\geq 0, \\
    \CVaR_{1-\epsilon}^{\mathbb{P}^*}(\CLC(x,\ubar{u},\boldsymbol{\xi}))\geq 0,
\end{align}
\end{subequations}
as the distributionally robust analogues of the CLF and CBF conditions from~\cite{EDS:98} and~\cite{ADA-SC-ME-GN-KS-PT:19}, resp. The existence of a controller satisfying~\eqref{eq:cvar-cbc-clc} implies the existence of a controller that makes the CLC (resp. the CBC) condition hold at every point with probability at least $1-\epsilon$,
paving the way for the design of controllers  that make the system stable (resp. safe) with arbitrarily high probability.

\section{Problem Statement}

Consider the system model in \eqref{eq:uncertain-model} with distributional uncertainty, meaning that the true distribution $\mathbb{P}^*$ of the parameter $\boldsymbol{\xi}$ is unknown. We assume that the system admits a CLF and a CBF, which allow us to formulate the constraints~\eqref{eq:cvar-cbc-clc}. Given a nominal controller specified by a smooth function $\ubar{k}:\real^n\to\ubar{\mathcal{U}}$, we would like to synthesize a controller closest to it that respects safety and stability constraints. Using \eqref{eq:cvar-convex}, this problem can be written in general form as
\begin{align}\label{eq:dr-general}
&\min\limits_{\ubar{u}\in\ubar{\mathcal{U}}} \norm{\ubar{u}-\ubar{k}(x)}^2 \\
\notag
\text{s.t.} &\sup_{\mathbb{P}\in\Mc_N^r} \inf_{t\in\real}[\epsilon^{-1}\mathbb{E}_{\mathbb{P}}[(G_l(x,\ubar{u},\boldsymbol{\xi})+t)_{+}]-t]\leq 0, \ \forall l\in[M],
\end{align}
where $M\in\integerspos$ and each $G_l:\real^n\times\ubar{\Uc}\times\Xi\to\real$ is an affine function in $\ubar{u}$ and $\boldsymbol{\xi}$,
$G_l(x,\ubar{u},\boldsymbol{\xi})=\ubar{u}^T q_l(x) + \ubar{u}^T R_l(x)\boldsymbol{\xi}$,
for smooth functions $q_l:\real^n\to\real^{m+1}$ and $R_l:\real^n\to\real^{(m+1)\times k}$. With $M=2$ and constraints corresponding to $\CBC$ and $\CLC$, this corresponds to a stable and safe control synthesis problem. The case $M=1$ with the constraint $\CBC$ corresponds to a distributionally robust version of a safety filter of $\ubar{k}$.

Although the constraints in~\eqref{eq:dr-general} are convex, the program is intractable due to the search of suprema over the Wasserstein set. Fortunately, \cite[Proposition IV.1]{KL-YY-JC-NA:23-acc} shows that when $\Xi=\real^k$ and $p=1$,
the following SOCP is equivalent to~\eqref{eq:dr-general}: 
\begin{subequations}
\begin{align}
&\min_{\ubar{u}\in\ubar{\Uc},y\in\real,t\in\real,s_i\in\real} y \\
\text{s.t.} & \quad r\norm{R_l^T(x)\ubar{u}} + \frac{1}{N}\sum_{i=1}^N s_i - t\epsilon \leq 0,\ \forall l\in[M],~\label{eq:subeq1-socc} \\
&\quad s_i \geq G_l(x,\ubar{u},\boldsymbol{\xi}_i) + t, \ \forall i\in[N], \ \forall l\in[M],~\label{eq:subeq:si-gl} \\
&\quad s_i\geq0, \ \forall i\in[N],~\label{eq:subeq1-si} \\
&\quad y+1\geq\sqrt{\norm{2(\ubar{u}-\ubar{k}(x))}^2 + (y-1)^2}~\label{eq:subeq1-y}.
\end{align}
\label{eq:dro-clf-cbf-socp}
\end{subequations}
We refer to \eqref{eq:dro-clf-cbf-socp} as the DRO-SOCP and take $\Xi=\real^k$ and $p=1$ Wasserstein distance throughout the paper.


A critical observation about problem~\eqref{eq:dro-clf-cbf-socp} is that, in general, it might be infeasible, leading to controllers that are undefined. 
    
Furthermore, even if the problem is feasible, the controller obtained from it might not be continuous, hence resulting in implementation problems (it might induce chattering behavior when implemented on physical systems) and theoretical problems (lack of existence of solutions of the closed-loop system). 
Hence, our goal in this paper is twofold.
First, we derive conditions to ensure the feasibility of~\eqref{eq:dro-clf-cbf-socp}. Given the complexity of obtaining characterizations for the feasibility of such problems, we focus on identifying conditions that are easy to evaluate computationally as opposed to directly attempting to solve the optimization problem: either sufficient conditions, to quickly ensure feasibility, or necessary, to quickly discard it.
Second, assuming that the problem \eqref{eq:dro-clf-cbf-socp} is feasible, we characterize the regularity properties of the resulting controller.

\section{Feasibility Analysis}

In this section, we study the feasibility properties of~\eqref{eq:dro-clf-cbf-socp}. We start by giving a necessary condition for its feasibility.

\begin{proposition}\longthmtitle{Necessary condition for feasibility of DRO-SOCP}\label{prop:nec-cond-feasibility-1-constraint}
Let $\epsilon\in (0,\frac{1}{N}]$ and $r>0$. For $x\in\real^n$, let
\begin{align*}
    &\bar{Q}_l(x) = r R_l(x)_{2:(m+1)} \in \real^{m\times k}, \ \bar{r}_l(x) = r R_l(x)_1 \in \real^{1\times k}, \\
    &\bar{w}_{l,i}(x) = (-\epsilon q_l(x) - \epsilon R_l(x)\boldsymbol{\xi}_i)_{2:(m+1)} \in \real^m, \\
    &\bar{v}_{l,i}(x) = (-\epsilon q_l(x) - \epsilon R_l(x)\boldsymbol{\xi}_i)_1 \in \real, \\
    &\bar{F}_{l,i}(x) = \bar{Q}_l(x) \bar{Q}_l(x)^T-\bar{w}_{l,i}(x) \bar{w}_{l,i}(x)^T \in \real^{m \times m}, \\
    &\bar{J}_{l,i}(x) = \bar{r}_l(x) \bar{Q}_l(x)^T - \bar{v}_{l,i}(x)\bar{w}_{l,i}^T \in \real^{1\times m} , \\
    &\bar{H}_{l,i}(x) = \begin{pmatrix}
        \bar{r}_l \bar{r}_l^T - \bar{v}_{l,i}^2)(x) & \bar{J}_{l,i}(x) \\
        \bar{J}_{l,i}^T(x)  & \bar{F}_{l,i}(x)
    \end{pmatrix} \in \real^{(m+1)\times(m+1)}
\end{align*}
for $l\in[M]$ and $i\in[N]$. Let $\bar{\lambda}_{l,i}(x)$ be the minimum eigenvalue of $\bar{F}_{l,i}(x)$ and 
suppose $\bar{Q}_{l}(x)\bar{Q}_{l}(x)^T$ is invertible for all $l\in[M]$.
If~\eqref{eq:dro-clf-cbf-socp} is feasible,
then for each $l\in[M]$, there exists $i\in[N]$ such that $\bar{H}_{l,i}(x)$ is not positive definite and one of the following holds:
\begin{enumerate}
    \item $\bar{\lambda}_{l,i}(x)\! < \! 0$,
    \item $\bar{\lambda}_{l,i}(x) \! > \! 0$ and $\big( \bar{v}_{l,i}-\bar{w}_{l,i}^T \bar{F}_{l,i}^{-1}( \bar{Q}_{l} \bar{r}_{l}^T \! - \! \bar{w}_{l,i} \bar{v}_{l,i} ) \big)(x) \! \geq \! 0$,
    \item $\bar{\lambda}_{l,i}(x)=0$, and $\big(\bar{v}_{l,i}-\bar{w}_{l,i}^T(\bar{Q}_l\bar{Q}_l^T)^{-1}\bar{Q}_{l} \bar{r}_l^T \big)(x) > 0$.
\end{enumerate}
\end{proposition}
\begin{proof}
Note that~\eqref{eq:dr-general} (and hence~\eqref{eq:dro-clf-cbf-socp}) is equivalent to
\begin{align}\label{eq:dro-clf-cbf-socp-inf}
    &\min\limits_{\ubar{u}\in\ubar{\Uc}} \norm{\ubar{u}-\ubar{k}(x)}^2 \\
    \notag
    &\text{s.t.} \ r\norm{R_l^T(x)\ubar{u}} + \inf_{t\in\real} \Big[ \frac{1}{N}\sum_{i=1}^N (G_l(x,\ubar{u},\boldsymbol{\xi}_i)+t)_{+} -t\epsilon \Big] \leq 0,
\end{align}
for $l \in \until{M}$,
cf.~\cite[Proposition IV.1]{KL-YY-JC-NA:23-acc}.
For $(x,\ubar{u})\in\real^n\times\ubar{\Uc}$,
the function $A_{x,\ubar{u}}^l(t)=\frac{1}{N}\sum_{i=1}^N (G_l(x,\ubar{u},\boldsymbol{\xi}_i)+t)_{+} -t\epsilon$ is a piecewise linear function in $t$.
Since $\epsilon\leq\frac{1}{N}$,
it is decreasing for $t<t_l^*(x,\ubar{u}):=\min_{i\in[N]} -G_l(x,\ubar{u},\boldsymbol{\xi}_i)$
and increasing for $t>t_l^*(x,\ubar{u})$.
Hence, it achieves its minimum at $t_l^*(x,\ubar{u})$.
Thus,~\eqref{eq:dro-clf-cbf-socp} is feasible if and only if for all $l\in[M]$ the following inequalities are simultaneously feasible:
\begin{align}\label{eq:max-soc-data}
r\norm{R_l^T(x)\ubar{u}}+\epsilon\ubar{u}^T q_l(x)+\epsilon \max_{i\in[N]} \ubar{u}^T R_l(x) \boldsymbol{\xi}_i \leq 0.
\end{align}
Note that, 
if for some $l\in[M]$, the constraint $r\norm{R_l^T(x)\ubar{u}}+\epsilon \ubar{u}^T q_l(x)+\epsilon \ubar{u}^T R_l(x)\boldsymbol{\xi}_i\leq0$ is infeasible for all $i\in[N]$, then~\eqref{eq:dro-clf-cbf-socp} is infeasible.
Note that this is only a sufficient, but not necessary, condition for infeasibility (or equivalently, a necessary, but not sufficient, condition for feasibility).
The result follows from~\cite[Theorem 2]{FC-JJC-BZ-CJT-KS:21}, which characterizes the feasibility of a single second-order cone constraint.
\end{proof}

Next, we state a sufficient condition for the feasibility of \eqref{eq:dro-clf-cbf-socp} in the case $M=1$.

\begin{proposition}\longthmtitle{Sufficient condition for feasibility of DRO-SOCP with one constraint}\label{prop:suff-cond-one-constr-data}
Let $r > 0$,  $M=1$, and $0<\epsilon\leq\frac{1}{N}$.
Given $x\in\real^n$, define
\begin{align*}
    &\hat{Q}(x) = (r+\epsilon \max_{i\in[N]} \norm{\boldsymbol{\xi}_i} ) R_1(x)_{2:(m+1)} \in\real^{m\times k}, \\
    &\hat{r}(x) = (r+\epsilon \max_{i\in[N]} \norm{\boldsymbol{\xi}_i} ) R_1(x)_1 \in\real^{1\times k}, \\
    &\hat{w}(x) = -\epsilon q_1(x)_{2:(m+1)} \in\real^m, \ \hat{v}(x) = -\epsilon q_1(x)_{1} \in\real, \\
    &\hat{F}(x) = Q(x) Q(x)^T-w(x) w(x)^T \in\real^{m\times m}, \\
    &\hat{J}(x) = \hat{r}(x) \hat{Q}(x)^T - \hat{v}(x)\hat{w}(x)^T  \in \real^{1\times m}, \\
    &\hat{H}(x) = \begin{pmatrix}
         (\hat{r} \hat{r}^T - \hat{v}^2)(x) & \hat{J}(x) \\
        \hat{J}(x)^T & \hat{F}(x)
    \end{pmatrix} \in \real^{ (m+1)\times(m+1) }.
\end{align*}
Let $\hat{\lambda}(x)$ be the minimum eigenvalue of $\hat{F}(x)$. Suppose that $Q(x)Q(x)^T$ is invertible, $\hat{H}(x)$ is not positive definite and one of the following holds:
\begin{enumerate}
    \item $\hat{\lambda}(x)<0$,
    \item $\hat{\lambda}(x)>0$ and $\big( \hat{v}-\hat{w}^T \hat{F}^{-1}( \hat{Q} \hat{r}^T - \hat{w} \hat{v} \big)(x)\geq0$,
    \item $\hat{\lambda}(x)=0$ and $\big(\hat{v}-\hat{w}^T(\hat{Q}\hat{Q}^T)^{-1} \hat{Q} \hat{r}^T\big)(x)>0$.
\end{enumerate}
Then,~\eqref{eq:dro-clf-cbf-socp} is feasible at $x$.
\end{proposition}
\begin{proof}
By repeating an argument similar to the one in the proof of Proposition~\ref{prop:nec-cond-feasibility-1-constraint},~\eqref{eq:dro-clf-cbf-socp} is feasible in the case $M=1$ if and only if the following inequality is feasible:
\begin{align}\label{eq:max-m=1-socp}
    r\norm{R(x)^T \ubar{u}}+\epsilon\ubar{u}^T q(x)+\epsilon \max_{i\in[N]} \ubar{u}^T R(x) \boldsymbol{\xi}_i \leq 0.
\end{align}
Using the Cauchy-Schwartz inequality, the following inequality being feasible implies that~\eqref{eq:max-m=1-socp} is feasible,
\begin{align}\label{eq:max-norm-soc-suff-cond}
    (r+\epsilon \max_{i\in[N]}\norm{\boldsymbol{\xi}_i})\norm{R(x)^T \ubar{u} }+\epsilon \ubar{u}^T q(x) \leq 0.
\end{align}
If~\eqref{eq:max-norm-soc-suff-cond} is feasible, there exists $\hat{\ubar{u}}$ such that $r\norm{\hat{\ubar{u}}^T R(x)}+\epsilon\hat{\ubar{u}}^T q(x)+\epsilon \hat{\ubar{u}}^T R(x) \boldsymbol{\xi}_i \leq 0$ for all $i\in[N]$, and thus $\hat{\ubar{u}}$ satisfies~\eqref{eq:max-m=1-socp}.
The result follows by \cite[Thm.~2]{FC-JJC-BZ-CJT-KS:21}.
\end{proof}

\begin{remark}\longthmtitle{More data leads to better feasibility guarantees}\label{rem:more-data-better-feas}
     For a fixed $r$, the addition of new data points (larger $N$) implies that there are more chances that either of (i)-(iii) in Proposition~\ref{prop:nec-cond-feasibility-1-constraint} are satisfied for each $l \in \until{M}$.  Moreover, if $\mathbb{P}^*$ is light-tailed, $r_N(\bar{\epsilon})$ decreases with~$N$. The choice $r=r_N(\bar{\epsilon})$ means that for each fixed $i\in[N]$ and $l\in[M]$, the feasible set of the inequality $r\norm{R_l(x)^T \ubar{u}}+\epsilon \ubar{u}^T q_l(x)+\epsilon \ubar{u}^T R_l(x)\boldsymbol{\xi}_i\leq0$ increases,
     which from the proof of Proposition~\ref{prop:nec-cond-feasibility-1-constraint}, also means that there are more chances that either of (i)-(iii) are met. Similarly, under the assumption that the norm of additional samples is upper bounded by $\max_{i\in[N]}\norm{\boldsymbol{\xi}_i}$, the choice $r=r_{N}(\bar{\epsilon})$ also leads to a larger feasible set of~\eqref{eq:max-norm-soc-suff-cond}
    and thus the sufficient condition in Proposition~\ref{prop:suff-cond-one-constr-data} has more chances of being satisfied.
    \demo
\end{remark}

We next give a sufficient condition for the feasibility of~\eqref{eq:dro-clf-cbf-socp} with high probability for an arbitrary number of constraints.

\begin{proposition}\longthmtitle{Sufficient condition for feasibility of DRO-SOCP}\label{prop:suff-cond-feas-dro-socp}
    Let $r>0$, $\epsilon\in(0,1)$ and $\bar{\epsilon}\in(0,1)$.
    Suppose that there exists a controller $\hat{k}:\real^n\to\ubar{\Uc}$ and non-negative functions $S_l:\real^n\to\real_{\geq0}$ for $l\in[M]$ satisfying
    \begin{align}\label{eq:cvar-strict-feasibility}
        \CVaR_{1-\epsilon}^{\mathbb{P}^*}(G_l(x,\hat{k}(x),\boldsymbol{\xi})) &\leq -S_l(x), \quad \forall l\in[M].
    \end{align}
    Moreover, suppose that $\mathbb{P}^*$ is light-tailed and let $r_N(\bar{\epsilon})$ be defined as in~\eqref{eq:wasserstein_r_guarantee}.
    Let $x\in\real^n$ be such that $\norm{R_l(x)}\neq0$ for all $l\in[M]$, and let $B:\real^n\to\real_{\geq0}$ be an upper bound on the norm of $\hat{k}$.
    Then, if
    \begin{align}~\label{eq:suff-cond-feas-dro}
        r_{N}(\bar{\epsilon}) < \min\limits_{l\in[M]} \frac{\epsilon S_l(x)}{2\norm{R_l(x)}B(x)},
    \end{align}
    ~\eqref{eq:dro-clf-cbf-socp} is strictly feasible at $x$ with probability at least $1-\bar{\epsilon}$ for any $r\leq r_{N}(\bar{\epsilon})$.
\end{proposition}
\begin{proof}
Note that by definition, the first component of $\hat{k}(x)$ is $1$ for all $x\in\real^n$. Hence, $B(x)\geq\norm{\hat{k}(x)}\geq1$ for all $x\in\real^n$ so~\eqref{eq:suff-cond-feas-dro} is well-defined.
Let $t_1^*\in\real$ be such that
\begin{align*}
    \CVaR_{1-\epsilon}^{\mathbb{P}^*}(G_1(x,\ubar{u},\boldsymbol{\xi})) = \frac{1}{\epsilon}\mathbb{E}_{\mathbb{P}^*}[(G_1(x,\ubar{u},\boldsymbol{\xi})+t_1^*)_{+}] - t_1^*,
\end{align*}
and define $\hat{G}(x,\boldsymbol{\xi})=\frac{1}{\epsilon}(G_1(x,\hat{k}(x),\boldsymbol{\xi})+t_1^*)_{+} -t_1^*$.
Note that for any $\boldsymbol{\xi}, \boldsymbol{\xi^{'}}\in\real^k$,
\begin{align}\label{eq:lipschitzness}
\hspace{-0.3cm}
|\hat{G}(x,\boldsymbol{\xi}) \! - \!\hat{G}(x,\boldsymbol{\xi}^{'})|
\! \leq \! \frac{1}{\epsilon} \norm{R_1(x)} \! \cdot \! \norm{\hat{k}(x)} \! \cdot \! \norm{ \boldsymbol{\xi}-\boldsymbol{\xi^{'}}},
\end{align}
where we have used the fact that the operator $(\cdot)_{+}$ is Lipschitz with constant 1. Using~\eqref{eq:lipschitzness} in~\cite[Theorem 3.2]{PME-DK:18}, we conclude that for any $\mathbb{\hat{P}}\in\Pc_p(\Xi)$,
$|\mathbb{E}_{\mathbb{P}^*} \! (\hat{G}(x,\boldsymbol{\xi})) \! - \! \mathbb{E}_{\hat{\mathbb{P}}}(\hat{G}(x,\boldsymbol{\xi}))| \! \leq \! \frac{1}{\epsilon} \norm{R_1 (x)} \! \cdot \! \norm{\hat{k}(x)} \! \cdot \! W_1 (\mathbb{P}^* \! , \! \mathbb{\hat{P}})$.

From~\eqref{eq:suff-cond-feas-dro}, together with the fact that $\Mc_{N}^{r_{N}(\bar{\epsilon})}$ contains $\mathbb{P}^*$ with probability at least $1-\bar{\epsilon}$, cf. Remark~\ref{rem:choice-r}, and since the maximum Wasserstein distance between two distributions in $\Mc_N^{r_N(\bar{\epsilon})}$ is $2r_N(\bar{\epsilon})$, with probability at least $1-\bar{\epsilon}$,
\begin{align}\label{eq:s1-comparison}
\hspace{-0.2cm} | \! \CVaR_{1-\epsilon}^{\mathbb{P}^*}(G_1(x,\hat{k}(x),\boldsymbol{\xi})) \! - \! \mathbb{E}_{\hat{\mathbb{P}}}(\hat{G}(x,\boldsymbol{\xi}))| <  S_1(x).
\end{align}
for any  $\mathbb{\hat{P}}\in\Mc_{N}^{r_N(\bar{\epsilon})}$.
By definition of $\CVaR$, cf.~\eqref{eq:cvar-convex}, for any $\mathbb{\hat{P}}\in\Pc_p(\Xi)$, $\CVaR_{1-\epsilon}^{\hat{\mathbb{P}}}(G_1(x,\hat{k}(x),\boldsymbol{\xi}))\leq \mathbb{E}_{\hat{\mathbb{P}}}(\hat{G}(x,\boldsymbol{\xi}))$.
Combining this with~\eqref{eq:s1-comparison} and~\eqref{eq:cvar-strict-feasibility}, we get that with probability at least $1-\bar{\epsilon}$, $\CVaR_{1-\epsilon}^{\hat{\mathbb{P}}}(G_1(x,\hat{k}(x),\boldsymbol{\xi})) < 0$ for all
$\mathbb{\hat{P}}\in\Mc_{N}^{r_N(\bar{\epsilon})}$.
This argument holds for $l\in\{2,\hdots,N\}$, implying that $\hat{k}(x)$ is strictly feasible
for~\eqref{eq:dr-general} (and hence,~\eqref{eq:dro-clf-cbf-socp}) with probability at least $1-\bar{\epsilon}$ for any $r\leq r_{N}(\bar{\epsilon})$.
\end{proof}
\begin{remark}
\longthmtitle{Dependency of sufficient condition on slack terms}
Condition~\eqref{eq:cvar-strict-feasibility} on the controller $\hat{k}$ guarantees the satisfaction of the constraints in~\eqref{eq:dr-general} with a slack term $S_l(x)$ on the righthand side. Larger values of these slack terms mean that fewer samples are needed to satisfy \eqref{eq:suff-cond-feas-dro}.
Moreover, for the constraints in~\eqref{eq:cvar-cbc-clc},~\cite[Remark 2.3]{PM-JC:23-auto} 
shows how to obtain such functions $S_l$, even without knowledge of $\hat{k}$.
\demo
\end{remark}

\begin{remark}
\longthmtitle{Applicability of the sufficient condition}\label{rem:applicability-suff-cond}
    Checking condition~\eqref{eq:suff-cond-feas-dro} does not require precise knowledge of $\hat{k}$, just an upper bound of its norm. In particular, if bounds on the control norm are included as constraints in~\eqref{eq:dro-clf-cbf-socp}, those can be used to construct~$B$. Moreover, unlike Proposition~\ref{prop:suff-cond-one-constr-data}, condition~\eqref{eq:suff-cond-feas-dro} is agnostic to the samples $\{ \boldsymbol{\xi}_1,\hdots,\boldsymbol{\xi}_N \}$ and instead solely depends on its number~$N$. Note that for each $x\in\real^n$ with $\norm{R_l(x)}\neq0$ for all $l\in[M]$,
    if $S_l(x)>0$ for all $l\in[M]$, there exists $\hat{N}_x$ such that condition~\eqref{eq:suff-cond-feas-dro} holds for all $N\geq\hat{N}_x$. 
    This is because $r_{N}(\bar{\epsilon})$ is decreasing in $N$ and $\lim_{N\to\infty}r_N(\bar{\epsilon})=0$.
    The value $\hat{N}_x$ is state-dependent, larger for smaller values of $\epsilon$, $S_l(x)$ and larger values of $B(x)$. 
    \demo
\end{remark}


\begin{remark}\longthmtitle{Checking for (in)feasibility efficiently}\label{rem:complexity}
    A commonly used algorithm for solving SOCPs is the method in~\cite{MSL-LV-SB-HL:97}. For an SOCP with $r_{S}$ constraints and optimization variable of dimension $n_{S}$, it requires solving $\sqrt{r_{S}}$ linear systems of dimension $n_{S}$, and hence has
    complexity $\mathcal{O}(\sqrt{r_S}n_{S}^3)$, cf.~\cite{ALC:05}.
    Therefore,~\eqref{eq:dro-clf-cbf-socp} has complexity $\mathcal{O}(\sqrt{MN}(m+N)^3)$.
    Instead, since checking the positive definiteness of a symmetric matrix of dimension $n_{P}$ can be done by checking if its Cholesky factorization exists (which has complexity $\mathcal{O}(n_{P}^3)$), the complexity of checking the condition in Proposition~\ref{prop:nec-cond-feasibility-1-constraint} is $\mathcal{O}(NMm^3)$. Hence, for large $N$, it is much more efficient than solving the SOCP~\eqref{eq:dro-clf-cbf-socp} directly. We also note that the scaling in $M$ for the complexity of the SOCP solver is more favorable than that of checking the necessary condition.
    On the other hand, the complexity of checking the sufficient condition in Proposition~\ref{prop:suff-cond-one-constr-data} reduces to finding a maximum of $N$ numbers (which has complexity linear in $N$) and checking the positive definiteness of two symmetric matrices of dimension $m+1$ and $m$, resp. Hence, its complexity is $\mathcal{O}(N+m^3)$, which is also more efficient than solving the SOCP. Finally, note that the complexity of checking the conditions in Proposition~\ref{prop:suff-cond-feas-dro-socp} is constant in $N$ and $m$, and is linear in $M$ due to the minimum in~\eqref{eq:suff-cond-feas-dro}.
    Table~\ref{tab:summary} summarizes this complexity analysis.
    \demo
\end{remark}

\begin{table}[t!]
\centering
\caption{ Complexity of SOCP solver versus the results in this section.}\label{tab:summary}
\begin{tabular}{ |c|c|c|c| } 
 \hline
 Method & Necessary/Sufficient & Complexity & $M$ \\ 
 \hline
 \hline
 Prop.~\ref{prop:nec-cond-feasibility-1-constraint} & Necessary & $\mathcal{O}(NM m^3)$ & any \\ 
 \hline
 Prop.~\ref{prop:suff-cond-one-constr-data} & Sufficient & $\mathcal{O}(N+m^3)$ & $1$ \\ 
 \hline
 Prop.~\ref{prop:suff-cond-feas-dro-socp} & Sufficient & $\mathcal{O}(M)$ & any \\
 \hline
 SOCP solver & Necessary and sufficient & $\mathcal{O}( \sqrt{N M}(m+N)^3 )$ & any \\
 \hline
\end{tabular}
\vspace*{-3ex}
\end{table}
%
%

Proposition~\ref{prop:nec-cond-feasibility-1-constraint} provides necessary conditions for feasibility. If the conditions are not met, it is reasonable to gather more data for verifying feasibility without having to directly solve the program. Moreover, if the conditions in Propositions~\ref{prop:suff-cond-one-constr-data} and~\ref{prop:suff-cond-feas-dro-socp} are not met (which does not mean that~\eqref{eq:dro-clf-cbf-socp} is not feasible),  
this might be an indication that more data is needed to certify feasibility, cf. Remarks~\ref{rem:more-data-better-feas} and~\ref{rem:applicability-suff-cond}. 

\section{Regularity Analysis}\label{sec:regularity-analysis}
In this section, we show that the controller obtained by solving~\eqref{eq:dro-clf-cbf-socp} is point-Lipschitz.

\begin{proposition}\longthmtitle{Point-Lipschitzness of SOCP DRO}\label{prop:lipschitzness-socp}
Let $r>0$, $0 < \epsilon \leq \frac{1}{N}$ and suppose $R_l$ and $q_l$ are twice continuously differentiable for all $l\in[M]$. 
Let $\ubar{u}^*:\real^n\to\real^m$ be the function mapping $x\in\real^n$ to the solution of~\eqref{eq:dro-clf-cbf-socp} in $\ubar{u}$ at~$x$. If~\eqref{eq:dr-general} is strictly feasible at $x_0 \in \real^n$ (i.e., there exists a solution satisfying all the constraints strictly), then
 $\ubar{u}^*$ is point-Lipschitz at~$x_0$.
\end{proposition}

\begin{proof}
We first show the result for $M=1$.
Let
    $\Ic:=\argmax{i\in[N]}{G_1(x_0,\ubar{u}^*(x_0),\boldsymbol{\xi}_i)}$,
note that the set $\Ic$ is dependent on $x_0$, but we omit this dependency to simplify the notation.
Note also that since $G_1(x,\ubar{u},\boldsymbol{\xi}_i)$ is continuous in $x$ and $\ubar{u}$ for all $i\in[N]$, there exists a neighborhood $\Nc=\Nc_x\times\Nc_{\ubar{u}} \subset \real^n\times\ubar{\Uc}$ of $(x_0,\ubar{u}^*(x_0))$ such that for all $(\hat{x},\hat{\ubar{u}})\in\Nc$, there exists $i_{\hat{x},\hat{\ubar{u}}}\in\Ic$ such that $i_{\hat{x},\hat{\ubar{u}}}\in\argmax{i\in[N]}{G_1(\hat{x},\hat{\ubar{u}},\boldsymbol{\xi}_i)}$.
Recall from the proof of Proposition~\ref{prop:nec-cond-feasibility-1-constraint} that, for any $x,\ubar{u}\in\real^n\times\ubar{\Uc}$, 
the function $A_{x,\ubar{u}}(t):=\frac{1}{N}\sum_{i=1}^N (G_1(x,\ubar{u},\boldsymbol{\xi}_i)+t)_{+} -t\epsilon$
attains its minimum at $t^*(x,\ubar{u}):=\max_{i\in[N]}{G_1(x,\ubar{u},\boldsymbol{\xi}_i)}$.
Therefore, for $(\hat{x},\hat{\ubar{u}})\in\Nc$, $t^*(\hat{x},\hat{\ubar{u}})=G_1(\hat{x},\hat{\ubar{u}},\boldsymbol{\xi}_{i_{\hat{x},\hat{\ubar{u}}}})$.

For each $i\in\Ic$, let $\ubar{u}_i^*:\real^n\to\real^m$ be defined as:
\begin{align}\label{eq:dro-clf-cbf-i}
    \ubar{u}_i(x) := &\min\limits_{\ubar{u}\in\ubar{\Uc}} \norm{\ubar{u}-\ubar{k}(x)}^2
    \\
    \notag
    &\text{s.t.} \ r\norm{R_1(x)^T \ubar{u} } +\epsilon G_1(x,\ubar{u},\boldsymbol{\xi}_{i}) \leq 0.
\end{align}
Note that since~\eqref{eq:dr-general} is strictly feasible at $x_0$, there exists $\tilde{\ubar{u}}\in\ubar{\Uc}$ such that
$r\norm{ R_1(x_0)^T \ubar{\tilde{u}} } + \max_{i\in[N]} \epsilon G_1(x_0,\tilde{\ubar{u}},\boldsymbol{\xi}_i) < 0$.
By continuity of $R_1$ and $G_1$ in $x$, there exists a neighborhood $\tilde{\Nc}_x\subset\Nc_x$ of $x_0$ such that $r\norm{R_1(x)^T \ubar{\tilde{u}}} + \epsilon G_1(x,\tilde{\ubar{u}},\boldsymbol{\xi}_i) < 0$ for all $x\in\tilde{\Nc}_x$ and $i\in\Ic$.
This implies that~\eqref{eq:dro-clf-cbf-i} is strictly feasible for any $x\in\tilde{\Nc}_x$.
Hence, by~\cite[Proposition 5.4]{PM-JC:23-auto}, $\ubar{u}_i^*$
is point-Lipschitz at $x_0$ for each $i\in\Ic$.
Now, since for all $y\in\Nc_x$ there exists $i\in\Ic$ such that $\ubar{u}^*(y)=\ubar{u}_i^*(y)$, and $\tilde{\Nc}_x\subset\Nc_x$, it follows $\norm{\ubar{u}^*(y)-\ubar{u}^*(x_0)}=\norm{\ubar{u}_i^*(y)-\ubar{u}_i^*(x_0)}\leq \gamma_i \norm{y-x_0}$ for some $\gamma_i>0$.
Now, by taking $\gamma :=\max_{i\in\Ic} \gamma_i$, it follows that $\norm{\ubar{u}^*(y)-\ubar{u}^*(x_0)}\leq \gamma \norm{y-x_0}$ for all $y\in\Nc_x$ and hence $\ubar{u}^*$ is point-Lipschitz at $x_0$.
The argument if $M>1$ is analogous, defining a set $\Ic_l$ similar to $\Ic$ for each $l\in[M]$.
\end{proof}

Proposition~\ref{prop:lipschitzness-socp} implies in particular that $u^*$ is continuous at~$x_0$.
Note also that the strict feasibility assumption in Proposition~\ref{prop:lipschitzness-socp} is satisfied with a prescribed probability if the hypothesis of Proposition~\ref{prop:suff-cond-feas-dro-socp} is satisfied.

\section{Simulations}

\begin{figure*}
    \centering
    \subcaptionbox{\label{fig:1a}}
    {\includegraphics[width=0.33\linewidth]{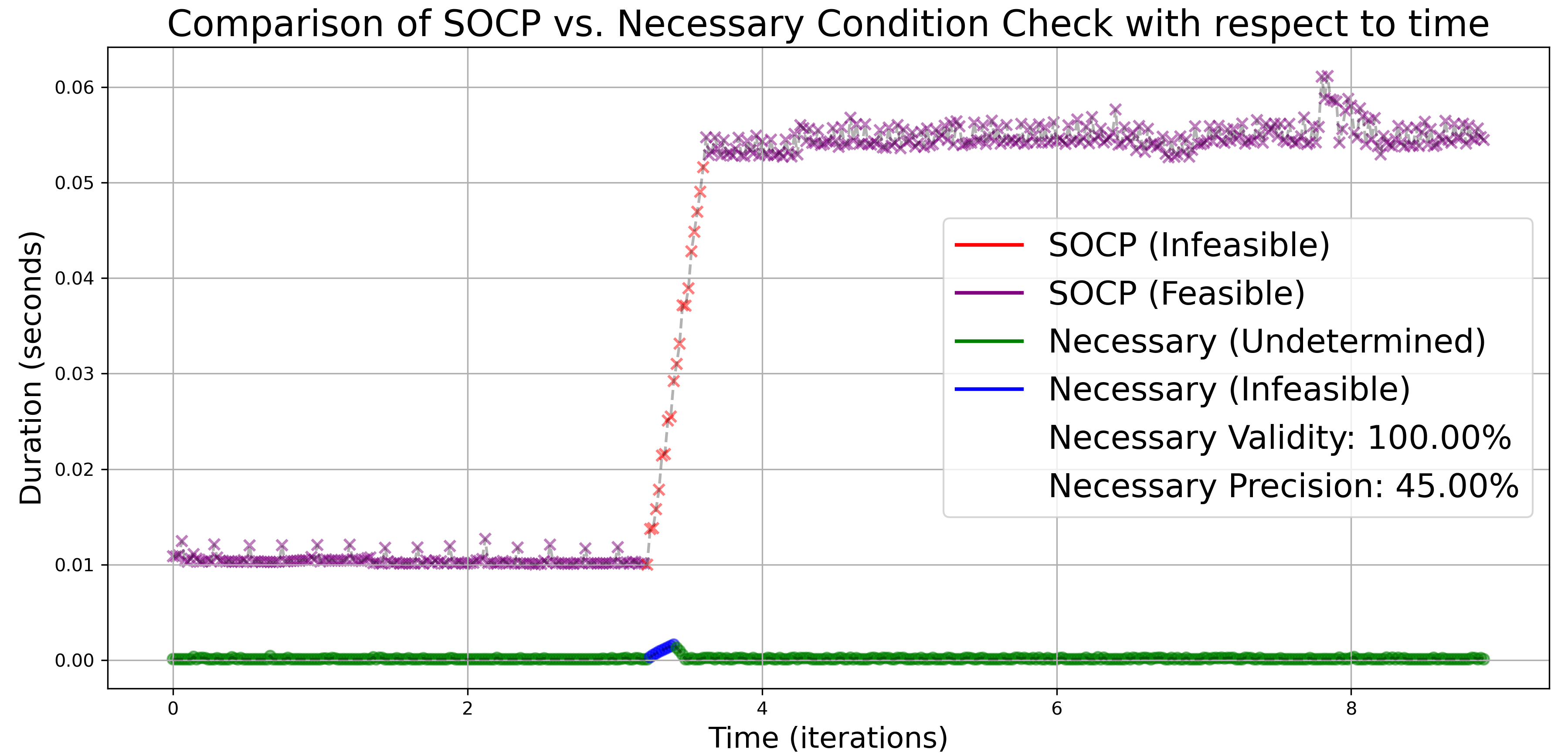}}%
    \subcaptionbox{\label{fig:1b}}
    {\includegraphics[width=0.33\linewidth]{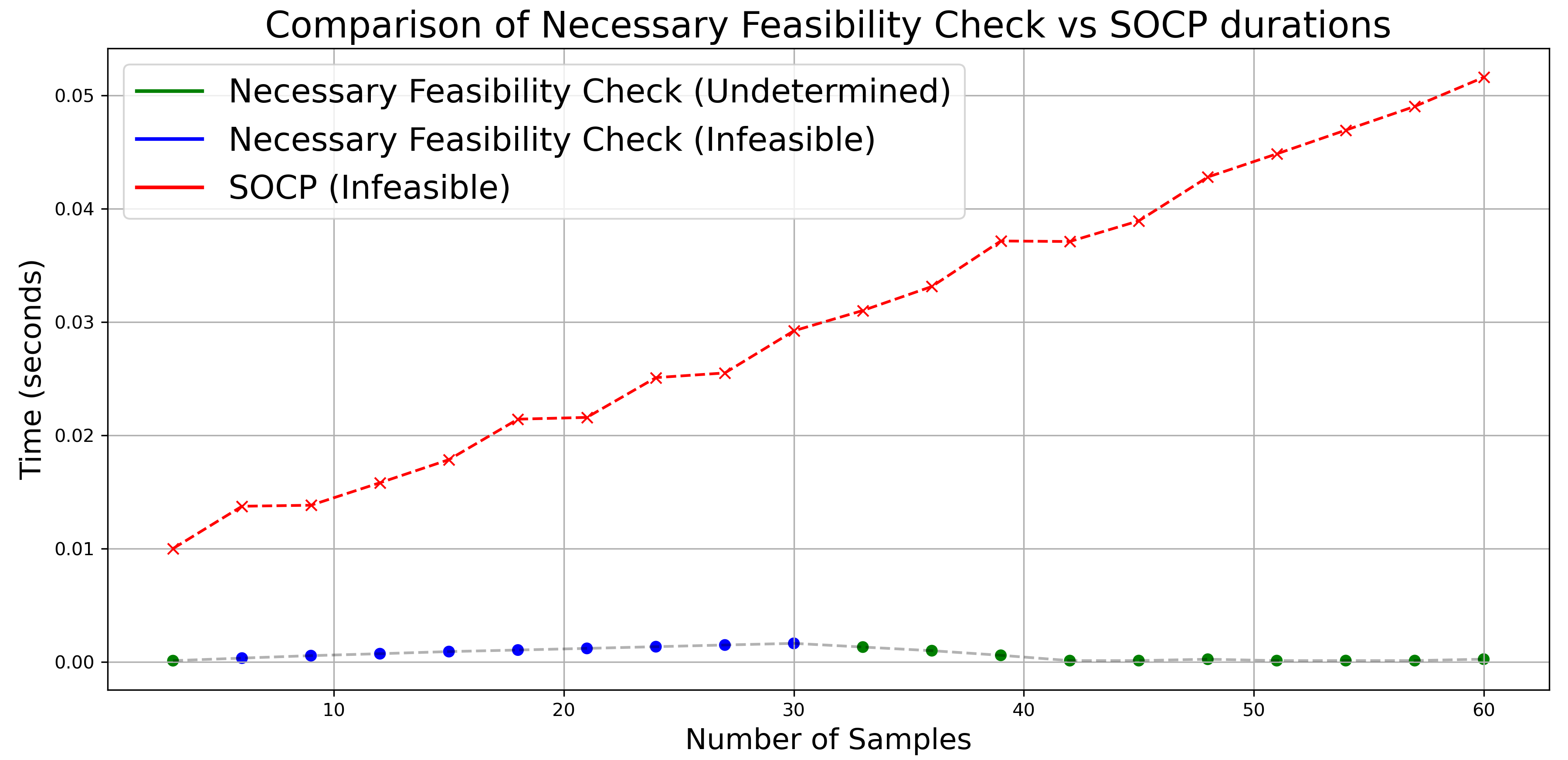}}%
    \subcaptionbox{\label{fig:1c}}
    {\includegraphics[width=0.33\linewidth]{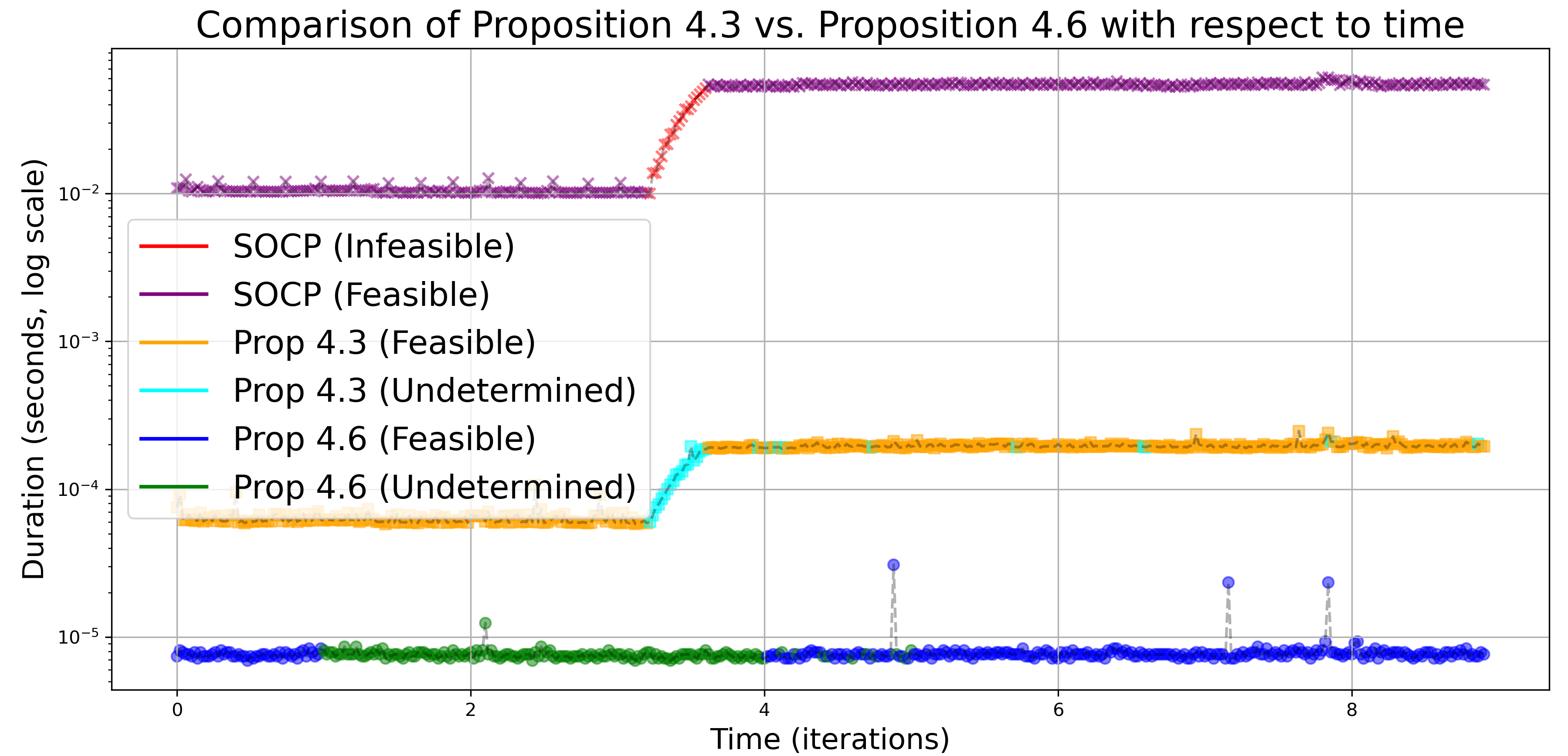}}
    \caption{(a): Time complexity comparison between necessary condition verification 
    (cf. Proposition~\ref{prop:nec-cond-feasibility-1-constraint}) and SOCP solver along the robot 
    trajectory. The label ``undetermined" means that the necessary condition is met, from which we may not know 
    if the problem is feasible or not. The label ``precision" represents the ratio of instances where the necessary 
    condition indicated infeasibility against the total number of instances where the SOCP 
    was actually infeasible. (b): Time complexity of necessary condition verification and 
    SOCP solver with increasing uncertainty samples (constraints). (c): Log-scaled time complexity comparison of two sufficient conditions (cf. Proposition~\ref{prop:suff-cond-one-constr-data} and Proposition~\ref{prop:suff-cond-feas-dro-socp}) with the SOCP solver along the robot trajectory.}
    \label{fig: time_compare}
    \vspace{-3ex}
\end{figure*}

In this section, we evaluate our results in a ground-robot navigation example. We model the robot motion using unicycle kinematics and take a small distance $a = 0.05$ off the wheel axis, cf.\cite{JC-ME:17-jcmsi} to obtain a relative-degree-one model:
\begin{equation*}
\label{eq:model_unicycle}
\begin{bmatrix} \dot{x}_1 \\ \dot{x}_2 \\ \dot{\theta} \end{bmatrix} = \Bigg(\begin{bmatrix} 0 & \cos(\theta) & -a\sin(\theta) \\ 0 & \sin(\theta) & a \cos(\theta) \\ 0 & 0 & 1  \end{bmatrix} + \sum_{j=1}^3 W_j(x)\xi_j) \Bigg) \begin{bmatrix} 1 \\ v \\ \omega\end{bmatrix},
\end{equation*}
where $v$, $\omega$ are the linear and angular velocity, and
\begin{equation*}
{\small
\begin{aligned}
	&W_1(x)
	\!=\!	
	\begin{bmatrix}
		0.02 & 0 & 0\\
		0.02 & 0 & 0\\
		0.01 & 0 & 0
		\end{bmatrix}\!\!,\, W_2(x) =
		\begin{bmatrix}
  		0 & 0 & 0\\
		0 & 0 & 0\\
		0 & 0 & -0.02
		\end{bmatrix}\!\!,\,  \\
        &W_3(x)
		\!=\!	
		\begin{bmatrix}
		0 & 0.02 \cos(\theta) & -0.02 a \sin(\theta) \\
		0 & 0.02 \sin(\theta) & 0.02 a \cos(\theta) \\
		0 & 0 & 0
		\end{bmatrix}\!\!,
\end{aligned}
}
\end{equation*}
represent the model perturbations in the drift, angular velocity, and orientation. We consider uncertainty samples: $\xi_1 \sim \mathcal{N}(0.5, 1)$, $\xi_2 \sim \mathcal{U}(-1, 1) $, and $\xi_3 \sim \mathcal{B}(2, 0.2)$, where $\mathcal{N}$, $\mathcal{U}$, $\mathcal{B}$ denote normal, uniform, and beta distributions, resp. The optimization programs are solved using the Embedded Conic Solver in CVXPY \cite{cvxpy} with an Intel i7 9700K CPU.

We first consider the problem of stabilizing the uncertain unicycle system to a goal position $[x_1^*, x_2^*] = [7,7]$ with initial state $[0,0,0]$,  so we take $M=1$ in~\eqref{eq:dr-general}. At the initial state, the robot is assumed to have $3$ samples $\{\boldsymbol{\xi}_i\}_{i=1}^3$ and initial Wasserstein radius $r=0.5$ with risk tolerance $\epsilon = 0.01$. As the robot moves, each unsuccessful solver attempt prompts the collection of additional  samples, and a~corresponding reduction in the ambiguity radius as prescribed by \eqref{eq:wasserstein_r_guarantee}. In all the figures presented, the $x$-axis represents the simulation timestep, where each timestep is equivalent to $0.02$ seconds, and the $y$-axis denotes the time spent for carrying out the necessary and sufficient condition checks, as well as for running the solver at each timestep.

The time complexity, validity, and precision of Proposition~\ref{prop:nec-cond-feasibility-1-constraint} are explored in Fig.~\ref{fig:1a} and Fig.~\ref{fig:1b}. Fig.~\ref{fig:1a} compares the time complexity of checking the necessary condition in Proposition~\ref{prop:nec-cond-feasibility-1-constraint} and of solving the corresponding SOCP along the whole robot trajectory. Notably, the SOCP becomes infeasible at around $t = 3$ s and more uncertainty samples are given until feasibility is regained. As expected, when Proposition~\ref{prop:nec-cond-feasibility-1-constraint} predicts the program is infeasible, such inference is consistently mirrored by the solver. Fig.~\ref{fig:1b} specifically emphasizes the time complexity during data collection stages. As the number of samples increases, the SOCP's time complexity escalates at a much faster rate than the necessary condition verification, in agreement with Remark~\ref{rem:complexity}. 

Fig.~\ref{fig:1c} compares the time complexity of solving the SOCP and of checking the sufficient conditions in Propositions~\ref{prop:suff-cond-one-constr-data} and~\ref{prop:suff-cond-feas-dro-socp}. As expected, feasibility validation by either result ensures the actual feasibility of the program by the solver. Checking Proposition~\ref{prop:suff-cond-one-constr-data} is more time-consuming than checking Proposition~\ref{prop:suff-cond-feas-dro-socp}, cf. Remark~\ref{rem:complexity},
but has greater accuracy in validating feasibility. Notably, both checks are significantly more efficient than solving the SOCP problem. 

We also consider the safe stabilizing problem of the unicycle system. The stabilization goal is $[x_1^*, x_2^*] = [5,5]$ while the safety goal is to avoid a circular obstacle centered at $[3,2]$ with radius $1$. Fig.~\ref{fig:m_2_plot} compares the time complexity and conservativeness of Proposition~\ref{prop:nec-cond-feasibility-1-constraint} and Proposition~\ref{prop:suff-cond-feas-dro-socp} for the case $M = 2$ in \eqref{eq:dr-general}. Proposition~\ref{prop:nec-cond-feasibility-1-constraint} is valid and requires significantly less time than solving the SOCP, while Proposition~\ref{prop:suff-cond-feas-dro-socp} is also valid and even more efficient.

\begin{figure}
    \centering
    \includegraphics[width = 0.48\textwidth]{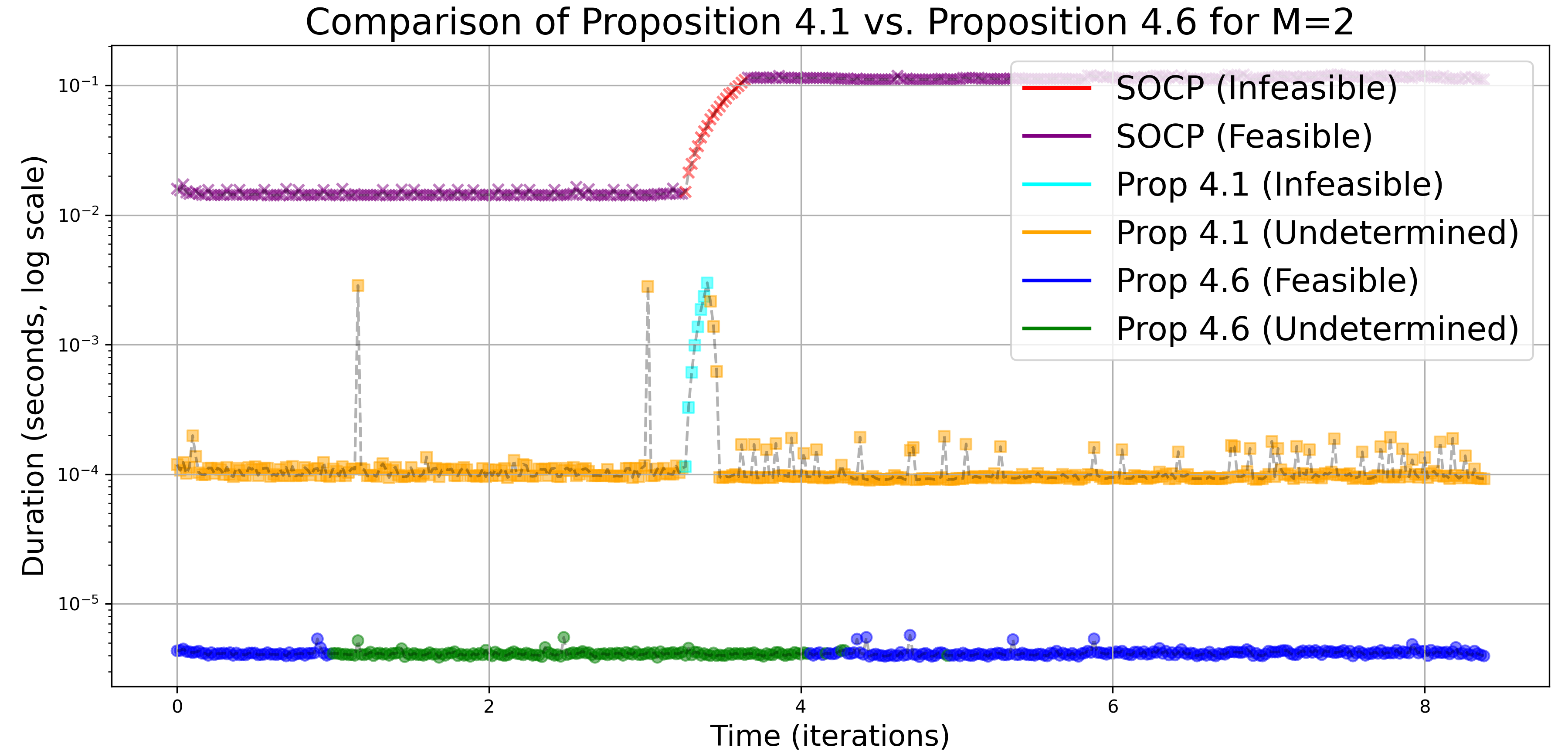}
    \caption{Time complexity comparison of necessary (cf. Proposition~\ref{prop:nec-cond-feasibility-1-constraint}) and sufficient (cf. Proposition~\ref{prop:suff-cond-feas-dro-socp}) conditions.}
    \label{fig:m_2_plot}
    \vspace{-3ex}
\end{figure}

\section{Conclusions}
We studied the feasibility of SOCP problems whose solution provide safe stabilizing controllers for uncertain systems with no prior knowledge of the uncertainty distribution and only a finite number of samples available. We provided a necessary condition and two sufficient conditions to check feasibility, characterized their computational complexity, and showed through simulations their usefulness in practical scenarios. We also showed that the controller obtained by solving the SOCP is point-Lipschitz under fairly general conditions. Future work will consider leveraging the identified feasibility conditions to guide online data-gathering policies that aim to reduce uncertainty about the system dynamics. 



\begin{thebibliography}{10}
\providecommand{\url}[1]{#1}
\csname url@samestyle\endcsname
\providecommand{\newblock}{\relax}
\providecommand{\bibinfo}[2]{#2}
\providecommand{\BIBentrySTDinterwordspacing}{\spaceskip=0pt\relax}
\providecommand{\BIBentryALTinterwordstretchfactor}{4}
\providecommand{\BIBentryALTinterwordspacing}{\spaceskip=\fontdimen2\font plus
\BIBentryALTinterwordstretchfactor\fontdimen3\font minus
  \fontdimen4\font\relax}
\providecommand{\BIBforeignlanguage}[2]{{%
\expandafter\ifx\csname l@#1\endcsname\relax
\typeout{** WARNING: IEEEtran.bst: No hyphenation pattern has been}%
\typeout{** loaded for the language `#1'. Using the pattern for}%
\typeout{** the default language instead.}%
\else
\language=\csname l@#1\endcsname
\fi
#2}}
\providecommand{\BIBdecl}{\relax}
\BIBdecl

\bibitem{KL-VD-ML-JC-NA:22-ral}
K.~Long, V.~Dhiman, M.~Leok, J.~Cort{\'e}s, and N.~Atanasov, ``Safe control
  synthesis with uncertain dynamics and constraints,'' \emph{IEEE Robotics and
  Automation Letters}, vol.~7, no.~3, pp. 7295--7302, 2022.

\bibitem{FC-JJC-BZ-CJT-KS:21}
F.~Casta{\~n}eda, J.~J. Choi, B.~Zhang, C.~J. Tomlin, and K.~Sreenath,
  ``Pointwise feasibility of {G}aussian process-based safety-critical control
  under model uncertainty,'' in \emph{{IEEE} Conference on Decision and Control},
  Austin, Texas, USA, 2021, pp. 6762--6769.

\bibitem{AJT-VDD-SD-BR-YY-ADA:21}
A.~J. Taylor, V.~D. Dorobantu, S.~Dean, B.~Recht, Y.~Yue, and A.~D. Ames,
  ``Towards robust data driven control synthesis for nonlinear systems with
  actuation uncertainty,'' in \emph{{IEEE} Conf.\ on Decision and Control},
  Austin, Texas, USA, 2021, pp. 6469--6476.

\bibitem{PM-JC:23-auto}
P.~Mestres and J.~Cort\'es, ``Feasibility and regularity analysis of safe
  stabilizing controllers under uncertainty,'' \emph{Automatica}, 2023,
  submitted.

\bibitem{AL-AB-ND-SH:23}
A.~Lederer, A.~Begzadic, N.~Das, and S.~Hirche, ``Safe learning-based control
  of elastic joint robots via control barrier functions,''
  arXiv 2212.00478, 2023.

\bibitem{PS-MJ-EH:22}
P.~Seiler, M.~Jankovic, and E.~Hellstrom, ``Control barrier functions with
  unmodeled input dynamics using integral quadratic constraints,'' \emph{IEEE
  Control Systems Letters}, vol.~6, pp. 1664--1669, 2022.

\bibitem{KL-YY-JC-NA:23-acc}
K.~Long, Y.~Yi, J.~Cort\'es, and N.~Atanasov, ``Safe and stable control
  synthesis for uncertain system models via distributionally robust
  optimization,'' in \emph{{A}merican {C}ontrol {C}onference}, Jun. 2023, pp. 4651--4658.

\bibitem{EDS:98}
E.~D. Sontag, \emph{Mathematical Control Theory: Deterministic Finite
  Dimensional Systems}, 2nd~ed., ser. TAM.\hskip 1em plus 0.5em minus
  0.4em\relax Springer, 1998, vol.~6.

\bibitem{ADA-SC-ME-GN-KS-PT:19}
A.~D. Ames, S.~Coogan, M.~Egerstedt, G.~Notomista, K.~Sreenath, and P.~Tabuada,
  ``Control barrier functions: theory and applications,'' in \emph{{E}uropean
  {C}ontrol {C}onference}, Naples, Italy, Jun. 2019, pp. 3420--3431.

\bibitem{ADA-XX-JWG-PT:17}
A.~D. Ames, X.~Xu, J.~W. Grizzle, and P.~Tabuada, ``Control barrier function
  based quadratic programs for safety critical systems,'' \emph{IEEE
  Transactions on Automatic Control}, vol.~62, no.~8, pp. 3861--3876, 2017.

\bibitem{FC-JJC-BZ-CJT-KS:21-acc}
F.~Casta{\~n}eda, J.~J. Choi, B.~Zhang, C.~J. Tomlin, and K.~Sreenath,
  ``{G}aussian process-based min-norm stabilizing controller for control-affine
  systems with uncertain input effects and dynamics,'' in \emph{{A}merican
  {C}ontrol {C}onference}, New Orleans, LA, May 2021, pp. 3683--3690.

\bibitem{KL-CQ-JC-NA:21-ral}
K.~Long, C.~Qian, J.~Cort\'es, and N.~Atanasov, ``Learning barrier functions
  with memory for robust safe navigation,'' \emph{IEEE Robotics and Automation
  Letters}, vol.~6, no.~3, pp. 4931--4938, 2021.

\bibitem{AB-LEG-AN:09}
A.~Ben-Tal, L.~E. Ghaoui, and A.~Nemirovski, \emph{Robust Optimization}, ser.
  Applied Mathematics Series. Princeton University Press, 2009.

\bibitem{PME-DK:18}
P.~M. Esfahani and D.~Kuhn, ``Data-driven distributionally robust optimization
  using the {W}asserstein metric: performance guarantees and tractable
  reformulations,'' \emph{Mathematical Programming}, vol. 171, no. 1-2, pp.
  115--166, 2018.

\bibitem{AN-AS:06}
A.~Nemirovski and A.~Shapiro, ``Convex approximations of chance constrained
  programs,'' \emph{SIAM Journal on Optimization}, vol.~17, no.~4, pp.
  969--996, 2006.

\bibitem{MSL-LV-SB-HL:97}
M.~S. Lobo, L.~Vandenberghe, S.~Boyd, and H.~Levret, ``Applications of
  second-order cone programming,'' \emph{Linear algebra and its applications},
  vol. 284, pp. 193--228, 1997.

\bibitem{ALC:05}
A.-L. Cholesky, ``Sur la résolution numerique des systèmes d'équations
  linéaires,'' \emph{Bulletin de la société des amis de la bibliothèque de
  l’École polytechnique}, vol.~39, 2005.

\bibitem{JC-ME:17-jcmsi}
J.~Cort\'es and M.~Egerstedt, ``Coordinated control of multi-robot systems: A
  survey,'' \emph{SICE Journal of Control, Measurement, and System
  Integration}, vol.~10, no.~6, pp. 495--503, 2017.

\bibitem{cvxpy}
S.~Diamond and S.~Boyd, ``{CVXPY}: A {P}ython-embedded modeling language for
  convex optimization,'' \emph{Journal of Machine Learning Research}, 2016.

\end{thebibliography}
\end{document}